\documentclass[a4paper,18pt]{article}
\usepackage{amsmath}
\usepackage{indentfirst} 
\usepackage[english]{babel}
\usepackage{graphics}
\usepackage{amssymb}
 \usepackage{fancyhdr}  
\usepackage{dsfont}
\usepackage{amsfonts} 
\usepackage{amsthm}
\usepackage[dvips]{graphicx}
\usepackage[latin1]{inputenc} 
\usepackage{enumerate}
\usepackage{amsxtra}
\usepackage{amstext}
\usepackage{amssymb} 
\usepackage{latexsym}
\usepackage[pdftex,colorlinks]{hyperref}

\newcommand{\N}{\mathbb N}

 \newcommand{\C}{\mathbb C}

\newcommand{\GG}{\mathbf{G}}

\newcommand{\R}{\mathbb R}
\newcommand{\PV}{\textrm{P}(V)}
\newcommand{\PVs}{\textrm{P}(V^*)}
\newcommand{\E}{\mathbb E}

\newcommand{\p}{\mathbb P}

\newcommand{\Z}{\mathbb{Z}}

\newcommand{\too}{\underset{n \to +\infty}{\longrightarrow}}

\newtheorem{prop}{Proposition}[section]
\newtheorem{lemma}[prop]{Lemma}

\newtheorem{theo}[prop]{Theorem}

\newtheorem{remark}[prop]{Remark}
\newtheorem{example}[prop]{Example}

 \title{ The central limit theorem for eigenvalues }
\author{\Large{Richard Aoun}\footnote{
  Richard~Aoun, \textsc{American University of Beirut, Department of Mathematics,  
Faculty of Arts and Sciences, 
 P.O. Box 11-0236 
Riad El Solh,
Beirut 1107 2020, 
LEBANON} 
  \textit{E-mail address}:  \texttt{ra279@aub.edu.lb}}} 
 \date{}
 \newcommand\shorttitle{The central limit theorem for eigenvalues}
 \newcommand\authors{Richard Aoun}
\begin{document}
 
 \maketitle
 
   \abstract{ We prove that the spectral radius of a strongly irreducible random walk on $\textrm{GL}_d(\R)$ (or more generally the vector of moduli of eigenvalues of a Zariski-dense random walk on a linear reductive group) satisfies a central limit theorem under an order two moment assumption.}\\\\
   
   \vspace{0.5cm}\noindent  \textbf{Keywords}:  Random matrix products,  Lyapunov exponents, Stationary measures, Central limit theorem\\
   \noindent  \textbf{AMS classification}:  60B15, 60F05, 37H15, 20P05.

 \tableofcontents

 \section{Statement of the results}
Let $V$ be a real vector space of dimension 
  $d\geq 1$ and $||\cdot||$ a norm on $V$. 
For simplicity of notation, the operator norm on $\textrm{End}(V)$ is also denoted by $||\cdot||$.  For every $g\in \textrm{GL}(V)$, we denote by $\rho(g)$ the spectral radius of $g$ and we let  $l(g):=  \max\{\ln^+ ||g|| ,  \ln^+ ||g^{-1}||\}$ with $x^+=\textrm{max}\{x,0\}$. If $\mu$ is a Borel probability measure on $\textrm{GL}(V)$, we say that $\mu$ has moment of order $i\in \N$ if
 $$\int{ l(g)^i \,d\mu(g)}<+\infty.$$
   The right (resp.~ left) random walk at time $n\in \N$ will be denote by $R_n=X_1\cdots X_n$
   (resp.~ $L_n=X_n \cdots X_1$) where $(X_i)_{i\in \N^*}$ is a family of independent and identically distributed $\textrm{GL}(V)$-valued random variables of law $\mu$. All our random variables will be defined on a probability  space $(\Omega, \mathcal{F}, \p)$ and the expectation operator is denoted by $\E$. \\

\noindent  When $\mu$ has a moment of order one, we denote by  $\lambda_1(\mu)$   the top Lyapunov exponent of $\mu$, i.e.~
 \begin{equation}\lambda_1(\mu)=\underset{n\rightarrow +\infty}{\lim} \frac{1}{n}{ \ln   ||L_n|| }.\label{eh}\end{equation}
This convergence holds almost surely and is a simple consequence of Kingman's subadditive ergodic theorem 
(it was first proved by Furstenberg and Kesten \cite{furstenberg-kesten} prior to Kingman's theorem).
 We denote by $\Gamma_{\mu}$ the semi-group generated by the support of $\mu$. 
 We say that $\Gamma_{\mu}$ is strongly irreducible if it does not stabilize a finite union of non-trivial subspaces of $V$.  \\

\noindent The convergence \eqref{eh} can be thought as a law of large numbers for the non-commutative random product $L_n$.
  A corresponding central limit theorem for $\ln ||L_n||$ has been established long ago under the assumption of strong irreducibility of $\Gamma_{\mu}$ and an exponential moment assumption on $\mu$
   (\cite{lepage}, \cite{goldsheid-guivarch}, \cite{bougerol-lacroix}). Recently, Benoist--Quint
   gave in  \cite[Theorem 1.1]{benoist-quint-tcllineaire} another proof of the CLT, which is valid under the optimal moment hypothesis: namely, that  of a moment of order two. 
   Our first main result gives the analogous statement for $\rho(L_n)$, namely

\begin{theo}
 Let $\mu$ be a probability measure on $\textrm{GL}(V)$ such that:
 
\begin{itemize}
\item $\mu$ has a moment of order two,
\item $\Gamma_{\mu}$ is strongly irreducible and has unbounded image in $\textrm{PGL}(V)$. \end{itemize} 
Then there exists $\sigma_{\mu}>0$ such that the following convergence in law  holds: 
 $$ \frac{  \ln \rho (L_n) - n \lambda_1(\mu)}{\sqrt{n}}\underset{n\rightarrow +\infty}{\overset{\mathcal{L}}{\longrightarrow}}
 \mathcal{N}(0, \sigma_\mu).$$
  \label{maintcl}\end{theo}

\noindent   This theorem is established in  \cite[Theorem 13.22]{benoist-quint-book} under a more restrictive exponential moment condition.
 The  main contribution of this note is to establish it under the optimal order two moment assumption. 
 
   \begin{remark} It will follow from the proof that the limit distribution $\mathcal{N}(0,\sigma_{\mu})$ is the same as the limiting distribution  of $\frac{\ln ||L_n|| - n \lambda_1}{\sqrt{n}}$. 
    \label{oop}\end{remark}
  
 \begin{remark}  As we show in 
  Example \ref{notconv},       the sequence of random variables 
 $ \frac{\ln \rho (L_n) - n \lambda_1(\mu)}{\sqrt{n}}$ may fail to converge in distribution if $\Gamma_{\mu}$ is not assumed to be strongly irreducible. 
  Note that even if it exists, the limit  distribution is not necessarily Gaussian as one can see by considering a random walk on diagonal matrices.  It is worth mentioning that, on the contrary,  $\ln \rho(L_n)$ satisfies  always a law of large numbers  (without any algebraic assumption on the support of $\mu$) 
  as is recently shown in \cite{aoun-sert}. \label{rem-strongly}\end{remark}

  \vspace{0.2cm}
 In view of the known CLT for   $\ln ||L_n||$,   the proof of Theorem \ref{maintcl} reduces to proving that $\frac{1}{\sqrt{n}} \ln\frac{\rho(L_n)}{||L_n||}$ converges in probability to zero when $\mu$ has a moment of order two. We will actually give   estimates of the  ratio $\frac{\rho(L_n)}{||L_n||}$ with only the assumption of a
  moment of order one  for $\mu$.  The main technical result of this note is therefore the following:\begin{theo}
Let $\mu$ be a probability measure on $\textrm{GL}(V)$ such that: 
\begin{itemize}
\item $\mu$ has a moment of order one,
\item $\Gamma_{\mu}$ is strongly irreducible. \end{itemize}
Then
\begin{equation}\limsup_{n \to +\infty} \p\left(\frac{\rho(L_n)}{||L_n||} \leq \epsilon\right) \underset{\epsilon \to 0}{\big\downarrow} 0.\label{ratioe}\end{equation}
Equivalently, 
  for every numerical sequence $(\epsilon_n)_{n\in \N^*}$ that tends to zero, 
 \begin{equation}\p \left( \frac{\rho(L_n)}{||L_n||} \leq \epsilon_n\right) \too 0.\label{equamain}\end{equation}
\label{main}\end{theo}

\begin{remark}The speed of convergence  when $\epsilon\to 0$ in \eqref{ratioe} depends on the regularity of the unique stationary probability measure   on the projective space of some strongly irreducible and proximal representation. This is formulated   in  Theorem \ref{mainprecise} which is  a more precise statement than the one given in Theorem \ref{main}.  \end{remark}
In a similar fashion,  using the various wedge power representations of $\textrm{SL}_d(\R)$, 
 we can deduce easily from Theorem \ref{main} and from Benoist--Quint's central limit for the Cartan projection a CLT for the full vector of eigenvalues. Namely:

\begin{theo}
 Let $\mu$ be a probability measure on $\textrm{SL}_d(\R)$ with a moment of order $2$. Assume that its support generates a 
 Zariski-dense subgroup. Then there is a positive definite quadratic form $K_\mu$ on the hyperplane $\{(x_1, \cdots, x_d)\in \R^d; \sum_1^d x_i = 0\} $ such that the random vector 
 $$\left(\rho_1(L_n) - n \lambda_1(\mu), \cdots, \rho_d(L_n) - n \lambda_d(\mu)\right)/\sqrt{n} $$
 converges in law to the multidimensional gaussian centered gaussian distribution $\mathcal{N}(0,K_\mu)$. Here  $\rho_1(L_n) \geq  \cdots \geq \rho_d(L_n)>0$ denote  the moduli of the eigenvalues of $L_n$
  in decreasing order, and $\lambda_i(\mu):=\underset{n\rightarrow +\infty}{\lim}\frac{1}{n}\ln\frac{||\bigwedge ^ i L_n||}{||\bigwedge ^ {i-1} L_n||}$ is the $i$-th Lyapunov exponent.
\label{maingen}\end{theo}

Recall that a subgroup of $ \textrm{SL}_d(\R)$ is called Zariski-dense if it is contained in no proper real algebraic subgroup of $\textrm{SL}_d(\R)$. Theorem \ref{maingen} will be proved in a more general setting (see Theorem \ref{reductive}), that of random walks on Zariski dense sub-semigroups of  reductive groups.  \begin{remark} 
Under the assumption of Theorem \ref{maingen},  we have  $\lambda_1(\mu)>\cdots >\lambda_d(\mu)$ as it follows from the combination of 
Guivarc'h--Raugi's theorem \cite{guivarch-raugi} and Goldsheid--Margulis's one \cite{goldsheid-margulis} (see also Benoist--Labourie \cite{benoist-labourie} and Prasad \cite{prasad-rapinchuk}). In particular, by Theorem \ref{maingen},  all the eigenvalues of $L_n$ are real with a probability tending to one.  
 \end{remark}
  \begin{remark} 
  Theorems \ref{main}, \ref{maintcl}  and \ref{maingen} are also valid for $R_n$ verbatim, since $R_n$ and $L_n$ have the same law for every $n$.
 \end{remark}
 
   \begin{remark}
  When the field $\R$ is replaced by another local field (as $\C$ or a p-adic field for instance), Theorem \ref{main} remains true verbatim with the same proof. The limiting distributions provided in Theorem \ref{maintcl} and \ref{maingen} exist, remain Gaussian but can be however degenerate.  See \cite[Example 12.21]{benoist-quint-book} for more details.\end{remark}
  
  \section*{Acknowledgements}
\noindent The author has the    pleasure to thank Emmanuel Breuillard and {\c Ca$\check{\text{g}}$r\i }              
Sert  for fruitful discussions. 
Part of this project was sponsored by the Center of Advanced Mathematical Sciences (CAMS). 
 
\section{Preliminary reduction}\label{prel}
\noindent In this section, we reduce the proofs of Theorem \ref{main} and   Theorem \ref{maintcl} to  Theorem \ref{key} below which says essentially that the    attracting point of $L_n$ is fairly far from its  repelling hyperplane.\\

\noindent First, we introduce some notation.   We set $V=\R^d$. Let $\textrm{P}(V)$ be the projective space of $V$. For every non zero vector $v$ (resp.~ non zero subspace $E$) of $V$, we denote by $[v]=\R v$ (resp.~$[E]$) its projection to $\textrm{P}(V)$. The action of $g\in \textrm{GL}(V)$ on a vector $v$ will be simply denoted by $g v$, while the action of $g$ on a point $x \in \textrm{P}(V)$ will be denoted by $g\cdot x$. \\

 \noindent We endow $V$ with the canonical basis $(e_1, \cdots, e_d)$  and the usual Euclidean dot product and norm.    Let $K=O_d(\R)$ be the orthogonal group.   Denote by  $A\subset \textrm{GL}_d(\R)$ the subgroup of diagonal matrices and $A^+ \subset A$ the sub-semigroup made of matrices with positive entries and arranged in decreasing order. The KAK decomposition (or the singular value decomposition) states that   $\textrm{GL}_d(\R)=K A^+ K$.    For every $g\in \textrm{GL}(V)$, we denote by $g=k_g a(g) u_g$ a KAK decomposition of $g$ in the basis $(e_1,\cdots,e_d)$ of $V$. We write $a(g)=\left(a_1(g), \cdots, a_d(g)\right)$ with $a_1(g)\geq \cdots \geq a_d(g)>0$.  Even though $k_g$ and $u_g$ are not uniquely defined, we can always fix once for all a privileged choice of a KAK decomposition. We call \emph{attracting point } and \emph{repelling hyperplane} the following respective point in $\textrm{P}(V)$ and projective hyperplane of $\textrm{P}(V)$: 
 $$v_g^+=k_g[e_1] \,\,\,,\,\,\,  H_g^{-} =  [\ker (u_g^{-1} e_1^*)]= \left(\R \,(u_g^{-1} e_1)\right)^{\perp}.$$
 In the definitions above, $(e_1^*, \cdots, e_n^*)$ denotes the dual basis of $(e_1,\cdots, e_n)$ in the dual vector space $V^*$ of $V$. Also $\textrm{GL}(V)$ acts on $V^*$ by $(gf)(x)=f(g^{-1} x)$, $g\in \textrm{GL}(V)$, $f\in V^*$ and $x\in V$. \\

 \noindent Endow the vector space $\bigwedge^2 V$ with the canonical norm associated to the basis $(e_i \wedge e_j)_{1\leq i<j\leq d}$. 
We  endow $\textrm{P}(V)$ with the standard metric $\delta$ defined by: 
$$\forall x=[v], y=[w]\in \textrm{P}(V), \delta(x,y):=\frac{|| v \wedge w||}{||v||\,||w||}.$$
This is just the sine of the angle between the lines $x= \R v$ and $y=\R w$. 
\noindent  Finally, an endomorphism $g\in \textrm{End}(V)$ is said to be \emph{proximal} if it has a unique eigenvalue with maximal modulus and a  sub-semigroup $\Gamma$ of $\textrm{GL}(V)$ is said to be proximal if it contains a proximal element. \\

\noindent We are now able to  state our main technical result: 

 \noindent 
 
 \begin{theo} Assume that $\mu$ has a moment of order one  and that $\Gamma_{\mu}$ is strongly irreducible and proximal.
 Then 
 $$\limsup_{n \to +\infty} \p\left(\delta(v_{L_n}^+, H_{L_n}^{-}) \leq  \epsilon\right) \underset{\epsilon \to 0}{\big\downarrow} 0.$$
 
 \noindent Equivalently, 
   for any sequence of real numbers $(\epsilon_n)_n$ such that $\epsilon_n \underset{n\rightarrow +\infty}{\longrightarrow}  0$,  $$\p \left( \delta(v_{L_n}^+, H_{L_n}^{-}) \leq \epsilon_n\right)  { \underset{n\rightarrow +\infty}{\longrightarrow} } 0.$$
 \label{key}\end{theo}

In order to deduce Theorem \ref{main} from Theorem \ref{key}, we need the   following geometric lemma.  It  is    borrowed from Benoist-Quint \cite[Lemma 13.14]{benoist-quint-book}. For the convenience of the reader, we include a proof. 
\begin{lemma}
Let $g\in \textrm{GL}(V)$.
  If $\delta(v_g^+, H_g^{-})> 2 \sqrt{\frac{a_2(g)}{a_1(g)}}$, then 
 $$  \frac{\rho(g)}{||g||} \geq   \frac{\delta(v_g^+, H_g^{-})}{2}.$$
\noindent Moreover, in this case, $g$ is necessarily a proximal element. 
\label{algebra}\end{lemma}

 \begin{proof}
Fix $g\in \textrm{GL}(V)$. To simplify the notation, let  $\delta_g:= \delta(v_g^+, H_g^{-})$. For every $\epsilon>0$,   
let $U_{\epsilon}\subseteq \textrm{P}(V)$ be the complement of the  closed $\epsilon$-neighborhood around $H_g^{-}$, i.e.~
 $U_{\epsilon}:=\{x\in \textrm{P}(V); \delta(x,H_g^{-}) >  \epsilon\}$. 
The following statements are easy to verify using the definition of the Cartan decomposition and the standard metric $\delta$ 
 (except the   statement i.~ which directly    follows   from the   triangle inequality)
\begin{enumerate}
\item[i.] $B(v_g^+, \delta_g/2) \subset  U_{\delta_g/2}$, where $B(x,r)$ refers to the open ball of center $x\in \textrm{P}(V)$ and radius $r$ in the metric space $\left(\textrm{P}(V), \delta\right)$. 
\item[ii.] For every $\epsilon>0$, and for $r:= \left(  1+ \frac{a_1(g)^2}{a_2(g)^2} \epsilon^2 \right)^{-1/2}$, one has that 
$$g \cdot U_{\epsilon}  \subseteq \overline{B(v_g^+, r)}\subset B\left(v_g^+, \frac{a_{2}(g)}{a_{1}(g)} \, \frac{1}{\epsilon}\right).$$
\item[iii.] \begin{equation}\forall v\in V\setminus \{0\},  \left[\delta([v],H_g^{-})\right]^2\leq \left(\frac{||gv||}{||g||\,||v||} \right)^2 \leq  \left[\delta([v],H_g^{-})\right]^2+  \left(\frac{a_2(g)}{a_1(g)}\right)^2.\nonumber
\end{equation}
In particular, 
\begin{equation}\forall  [v]\in U_{\epsilon},\, \frac{||gv||}{||g||\,||v||} \geq \epsilon.\label{sa1}\end{equation}
\item[iv.] For every $\epsilon>0$, $$\sup_{x,y \in U_{\epsilon}} \frac{\delta\left(g \cdot x,g \cdot y\right)}{\delta(x,y)} \leq \frac{a_{2}(g)}{a_{1}(g)} \frac{1}{\epsilon^2}.$$
 \end{enumerate}
 \noindent Since the family $(U_{\epsilon})_{\epsilon>0}$ is decreasing,  we deduce   from observations i.~ and ii.~ 
above that $\overline{g\cdot U_{\epsilon}}\subset U_{\epsilon}$  as soon as $\frac{a_{2}(g)}{a_{1}(g)} \frac{2}{\delta_g}\leq \epsilon \leq \frac{\delta_g}{2}$. 
    From now, we assume that $\delta_g> 2  \sqrt{\frac{a_{2}(g)}{a_{1}(g)}}$  and we set $\epsilon=\delta_g/2$. 
With these assumptions, inequality iv.~ implies that the action of $g$ on the complete metric space $\overline{U_{\epsilon}}$ is contracting. Thus $g$ has a unique fixed point $x_g^+$ in   $\overline{g \cdot U_{\epsilon}}\subset U_{\epsilon}$.   This fixed point provides  an  eigenvalue 
$\lambda$ of $g$ whose direction is given by $x_g^+$. By \eqref{sa1}, we have $\frac{ |\lambda|}{||g||} \geq \epsilon= \delta_g/2$. 
 A fortiori, the spectral radius $\rho(g)$ of $g$  satisfies the desired inequality.  This proves the desired lower bound. By Tits converse lemma (see e.g.~\cite[Lemma 4.7]{breuillard-strongtits}), we have that $g$ is proximal and that the unique fixed point $x_g^+$ of $g$ in $\textrm{P}(V)$ corresponding to the top eigenvalue  belongs to $U_{\epsilon}$. 
     \end{proof}

\begin{proof}[Proof of Theorem \ref{main} modulo Theorem \ref{key}:]  
First, we show that we can assume without loss of generality that $\Gamma_{\mu}$ is strongly irreducible and proximal (i-p to abbreviate). Indeed, let $p\in \{1, \cdots, d\}$ be the proximality index of $\Gamma_{\mu}$, i.e.~ the least integer $k\in \{1,\cdots, d\}$ such that there exists a sequence of scalars $\lambda_n\in \R$ and of elements $g_n\in \Gamma_{\mu}$ such that $\lambda_n g_n$ converges in $\textrm{End}(V)$ to a endomorphism of rank $k$.  By \cite[Lemma 4.13  ]{benoist-quint-tcllineaire} there exist a $\Gamma_{\mu}$-invariant subspace $W$ of $\bigwedge^p V$ such that   such that the action of $\Gamma_{\mu}$ on $W$ is  i-p and such that   
$\{\frac{||g||^p}{||\pi(g)||}; g\in \Gamma_{\mu}\}$ is bounded, where $\pi: \Gamma_{\mu} \to \textrm{GL}(W)$ is the restriction representation\footnote{In positive characteristic, by \cite[Lemma 4.13]{benoist-quint-tcllineaire}, one has to replace $W$ by $W/U$ for some subspace $U$ of $W$ and $\pi$ by the representation on $W/U$. With these modifications, \eqref{toprox} remains true.} . 
 Let $C:=\sup\left\{\frac{||g||^p}{||\pi(g)||}; g\in \Gamma_{\mu}\right\}\in [1,+\infty)$. 
Since $\rho\left(\pi(g)\right)\leq \rho(\wedge ^p g)\leq \rho(g)^p$,   then for every $\epsilon>0$ and $n\in \N$, 
\begin{equation}\p\left(\frac{\rho(L_n)}{||L_n||}\leq \epsilon\right)\leq \p\left(\frac{\rho(\pi(L_n))}{||\pi(L_n)||}\leq C\epsilon^p \right).\label{toprox}\end{equation}
Thus  proving Theorem \ref{main} for $\Gamma_{\pi(\mu)}$ is enough to prove the same estimate for $\Gamma_{\mu}$. \\

\noindent For now on $\Gamma_{\mu}$ is assumed to be i-p. 
For every $n\in \N$,  let $L_n=k_n a_n u_n$   be a  KAK decomposition of $L_n$, $v_n^+$ the attracting point of $L_n$, $H_n^-$ its repelling hyperplane and     $\delta_n:=\delta(v_n^+, H_n^-)$. Let $\Omega_n\subseteq \Omega$ be the following event $$\Omega_n:=\left\{\omega\in \Omega; \,\delta_n^2 (\omega) > 4 \frac{a_{2,n}(\omega)}{ a_{1,n}(\omega)}  \right\}.$$
First, we check that \begin{equation}
\p(\Omega_n)\underset{n\rightarrow +\infty}{\longrightarrow} 1.\label{j1}\end{equation} Indeed, by definition of the Lyapunov exponents,   the following   convergence holds in probability: 
$$\frac{1}{n}\ln\frac{a_{2,n}}{a_{1,n}}  \underset{n\rightarrow +\infty}{\overset{\p}{{\longrightarrow}}} \lambda_2-\lambda_1.$$
 Hence  for $\gamma:= (\lambda_1(\mu) - \lambda_2(\mu))/2$,  
 $$\p\left(\frac{a_{2,n}}{a_{1,n}}\leq \exp(-n \gamma)\right) \underset{n\rightarrow +\infty}{\longrightarrow} 1.$$
Since $\Gamma_{\mu}$ is i-p,  Guivarc'h-Raugi's theorem \cite{guivarch-raugi} ensures that $\gamma>0$.     Applying now  Theorem \ref{key} for $\epsilon_n=3\exp(-n\gamma/2)$, gives that with probability tending to one, $\delta_n^2>9\exp(-n\gamma)>4\frac{a_{2,n}}{a_{1,n}}$, i.e.~$\p(\Omega_n) \underset{n\rightarrow +\infty}{\longrightarrow} 1$.  \\
Let now $\epsilon>0$. By Lemma \ref{algebra}, we have for every $n\in \N^*$,  
     $$\p\left( \frac{\rho(L_n)}{||L_n||} \leq \epsilon  \right) \leq  \p(\Omega\setminus \Omega_n) + \p \left( \delta(v_n^+, H_n^{-})\leq  2 \epsilon \right).$$
     Tending $n \rightarrow +\infty$ and using \eqref{j1}, we deduce that for every $\epsilon>0$, 
 \begin{equation}\limsup_{n\rightarrow +\infty}{\p\left( \frac{\rho(L_n)}{||L_n||} \leq  \epsilon  \right)}\leq  \limsup_{n\rightarrow +\infty}{\p \left( \delta(v_n^+, H_n^{-})\leq  2 \epsilon \right)}.\label{proofo0}\end{equation}
Applying    Theorem \ref{key}, we deduce that the quantity above converges to zero as $\epsilon \rightarrow 0$. 
  \end{proof}

 We easily deduce the proof of Theorem \ref{maintcl}. We recall the classical Slutsky's lemma in probability theory which asserts that if $(X_n)_{n\in \N}$ and $(Y_n)_{n\in \N}$ are two sequences of random  variables such that $(X_n)_{n}$ converges in law to a random variable $X$ and $(Y_n)_{n\in \N}$ converges in probability to a constant $c\in \R$, then the joint vector $(X_n,Y_n)$ converges in law to $(X,c)$; a fortiori $X_n+Y_n$ converges in law towards $X+c$. 

\begin{proof}[Proof of Theorem \ref{maintcl}:]
 By Benoist--Quint's central limit theorem \cite[Theorem 1.1]{benoist-quint-tcllineaire} for $\ln ||L_n||$ and   Slutsky's lemma,   all we need   to show is the following convergence in probability: 
\begin{equation}Y_n:=\frac{1}{\sqrt{n}} \ln \frac{||L_n||}{\rho(L_n)} \underset{n\rightarrow +\infty}{\overset{\p}{\longrightarrow}} 0.\label{i0}\end{equation}
This convergence is guaranteed by Theorem \ref{main} (using \eqref{equamain}).  \end{proof}

  We end this section by stating and proving a general version of   Theorem \ref{maingen}, using the language of reductive groups. 
  Before  stating the result, we recall standard notion of reductive groups (we refer for instance to \cite{Knapp}). Let $\GG$ be a linear reductive algebraic group assumed to be Zariski connected and denote by  
  $G=\GG(\R)$ its group of real points. We denote by $K$ a maximal compact subgroup of $G$, $\mathfrak{a}$ the Lie algebra of a 
  maximal $\R$-split torus $A$ with $\mathfrak{a}^+$ a positive Weyl chamber, 
  i.e.~ the cone in $\mathfrak{a}$ defined by the requirement that all positive roots take non-negative values. 
  Let $A^+=\exp\left(\mathfrak{a}^+\right)$. One has that $G=K A^+K$ called Cartan or KAK decomposition. The $A^+$-component of 
   an element of $G$ in this product is unique. 
  This yields the so-called Cartan projection $\kappa : G \longrightarrow \mathfrak{a}^+$. \\
  Recall also the Jordan decomposition:  any $g\in G$ can be written as a commuting product of a unipotent element, an elliptic element and a hyperbolic element (i.e.~an element with a conjugate in $A$). One can then define the    Jordan projection $\ell: G \longrightarrow \mathfrak{a}^+$ where 
  $\ell(g)$ is the unique element of $\mathfrak{a}^+$ such that $\exp(\ell(g))$ is conjugate to the hyperbolic part of $g$ in the 
  Jordan decomposition of $g$. These are projections of the linear group $G$ which encode the information of the moduli of eigenvalues and operator norms on certain linear representations of $G$.\\
  
\noindent  Let now  $\mu$ be a probability measure on $G$. We say that $\mu$ has a moment of order $p\geq 1$ if for some, or equivalently any, faithful linear representation   $\phi: G\longrightarrow \textrm{GL}_n(\R)$ of $G$, $\phi(\mu)$ has a moment of order $p$. Let   $(L_n)_{n\geq 1}$ be the left random walk on $G$ associated to $\mu$. The equivalent formulation of \eqref{eh} reads as follows:  when $\mu$ has a moment of order one, the vector $\frac{\kappa(L_n)}{n}$ converges almost surely to a non random  element 
$\overset{\rightarrow}{\lambda_{\mu}}\in \mathfrak{a}^+$,  called the Lyapunov vector of $\mu$. \\

\begin{theo}(Generalization of Theorem \ref{maingen}) Let  $\GG$ be a reductive real algebraic group, $G$ its group of real points,  
 $\Gamma$ a Zariski dense sub-semigroup of $G$. Consider a probability measure 
 $\mu$ on $\Gamma$ whose support generates $\Gamma$. 
 Assume that $\mu$ has a moment of order one. Then,  
 
\begin{equation}\limsup_{n \to +\infty} \p(\|\kappa(L_n)  - \ell(L_n)\| > M) \underset{M \to +\infty}{\big\downarrow} 0.\label{d1} \end{equation} 
\noindent Moreover, if $\mu$ has a moment of order two, then 
 the following convergence in law holds:
\begin{equation}\frac{ \ell(L_n) - n \overset{\rightarrow}{\lambda_{\mu}}}{\sqrt{n}} \underset{n\rightarrow +\infty}{\overset{\mathcal{L}}{\longrightarrow}} \mathcal{N}(0,K_{\mu}),\label{d2}
 \end{equation}
where $\mathcal{N}(0,K_\mu)$ is a multidimensional gaussian centered gaussian distribution. Its support is a vector subspace $\mathfrak{a}_\mu$ of $\mathfrak{a}$ which contains the intersection of $\mathfrak{a}$  with the Lie subalgebra of the derived group of $\mathbf{G}$. In particular $\mathfrak{a}_\mu=\mathfrak{a}$  when $\mathbf{G}$ is semisimple. 
 \label{reductive} \end{theo}
\begin{proof} 
Let $d$ be the real rank of $\GG$ and $d_S$ its semisimple rank. There exists a basis $\{\chi_1, \cdots, \chi_d\}$
 of the dual $\mathfrak{a}^*$ such that each $\chi_i$ is a highest weight of some irreducible  representation $V_i$ of $\GG$ (all of these representations are also    strongly irreducible by Zariski connectedness of $\GG$). Indeed, it is enough to concatenate   the $d_S$ fundamental weights to  $d-d_S$ characters of the  abelianization $\GG/[\GG,\GG]$ of $\GG$.    But if $(\psi,V)$ is an irreducible representation of $\GG$ and $\chi$ is a highest weight, then by using Mostow's theorem \cite{mostow}, we can find a norm $||\cdot||$  (depending on $\psi$ and $V$) on each $V$ such that for every $g\in \GG(\R)$, 
 $\chi(\kappa (g))=\ln (||\psi(g)||)$ and $\chi(\ell(g))=\ln \left(\rho(\psi(g))\right)$. Applying now Theorem \ref{main} on each $(V_i,\rho_i)$   proves \eqref{d1}.\\
We deduce \eqref{d2} from \eqref{d1} in the same way we deduced Theorem \ref{maintcl} from Theorem \ref{main}, i.e.~using \eqref{d1}, Slutsky's lemma and Benoist--Quint's central limit theorem for the Cartan projection \cite[Theorem 4.16]{benoist-quint-tcllineaire}. 
\end{proof}

 \begin{remark}  
  The Lyapunov vector lies actually in the open Weyl chamber $\mathfrak{a}^{++}$. This is well-known and follows from the combination of  Guivarc'h--Raugi's theorem \cite{guivarch-raugi} on the simplicity of the Lyapunov spectrum together with Goldsheid--Margulis's result \cite{goldsheid-margulis}  (see also Benoist--Labourie \cite{benoist-labourie} and Prasad \cite{prasad-rapinchuk}) concerning the existence of proximal elements in Zariski dense subgroups of real algebraic groups with proximal elements. \end{remark}
  
  \begin{remark}
 Assume now that $\mu$ has a moment of order two.   Breuillard and Sert refined recently the previous result. Indeed, they proved   in \cite[Theorem 1.9]{breuillard-cagri} that $\overset{\rightarrow}{\lambda_{\mu}}$
   lies in the interior of the Benoist cone of $\Gamma$ introduced by Benoist \cite{benoistcone} (it is the closure in $\mathfrak{a}^+$ of the positive linear combinations of $\ell(g)$, $g\in \Gamma$).  
    \end{remark}

\section{Estimates with a moment of order one}
In this section,  we provide  some qualitative estimates concerning the behavior of the random walk with a moment of order one.  In particular, item 5 of Proposition \ref{prop-estimates} is the analog of   Theorem \ref{key}, but with    $H_{L_n}^{-}$
   replaced by a deterministic hyperplane $H$. The uniformity of our constants in the hyperplane $H$ will be crucial as showed in the next section, where we fully prove Theorem \ref{key}. The regularity of the stationary measure  on   projective space (Lemma \ref{p0} below) is a crucial ingredient.   We also discuss what is known about the speed of convergence in Remark \ref{r11} and Remark \ref{r12}, when $\mu$ has higher order moments. \\

     \noindent We recall that if $\mu$ is a probability measure on $\textrm{GL}_d(\R)$, then a probability measure $\nu$ on $\PV$ is said to be $\mu$-stationary if for every continuous function $f$ on $\textrm{P}(V)$, $\int_{\textrm{P}(V)}{f  d\nu}=\iint_{\textrm{GL}(V)\times \textrm{P}(V)}{f(g \cdot x) d\mu(g) d\nu(x)}$. \\
    When $\Gamma_{\mu}$ is strongly irreducible, classical arguments of Furstenberg show that any $\mu$-stationary probability measure on $\PV$ is non-degenerate, i.e.~ $\nu(H)=0$ for every projective hyperplane $H$ (see \cite[Chapter III, Proposition 2.3]{bougerol-lacroix}). Now, if $\Gamma_{\mu}$ is strongly irreducible and  proximal, Guivarc'h and Raugi proved in \cite{guivarch-raugi} that there exists a unique $\mu$-stationary probability measure $\nu$ on $\textrm{P}(V)$. The next lemma gives information on the regularity of stationary measures.

\begin{lemma} Assume that $\mu$ has a moment of order one and that  $\Gamma_{\mu}$ is strongly irreducible. 
Let $\nu$ be any $\mu$-stationary probability measure on $\textrm{P}(V)$. Then,  
\begin{equation}\sup_{\textrm{$H$ projective hyperplane of $\textrm{P}(V)$}} \,\nu\{x\in \textrm{P}(V); \delta(x,H)\leq t\}   \underset{t\rightarrow 0}{\longrightarrow} 0.
\label{es1}\end{equation}
Equivalently, there exists a proper map $\phi: [1,+\infty) \longrightarrow [1,+\infty)$ such that 
\begin{equation}\sup_{\textrm{$H$ projective hyperplane of $\textrm{P}(V)$}}\, \int_{\textrm{P}(V)} {\phi\left(\frac{1}{\delta(x,H)} \right)\,d\nu(x)} < +\infty.\label{es2}\end{equation}
\label{p0}\end{lemma}

\begin{proof} Let $\nu$ be a $\mu$-stationary probability measure on $\PV$.   If \eqref{es1} was not true, then there would exist some $\epsilon_0>$ and a sequence of projective hyperplanes $H_n:=[\ker(f_n)]$ such that for every $n\in \N$, 
\begin{equation}\nu\left\{x\in \PV; \delta(x,H_n)\leq \frac{1}{n}\right\}>\epsilon_0.\label{hyp1}\end{equation} 
  Since $\PVs$ is compact, we can extract a convergent subsequence $[f_{n_k}]_{k\in \N}$ of $[f_n]$, say to $[f] \in \PVs$. Let $H:=[\ker(f)]$ and $a_n=\frac{1}{n}+\sqrt{2}\delta([f_n],[f])$. Observe that the following inequality is true for every $x\in \PV$ and every $f,f'\in V^*$ such that $||f||=||f'||=1$,  \begin{equation}\label{useful} \Big| \delta([x], [\ker f])-\delta([x],[\ker f']) \Big|=\Big| \,\frac{|f(x)|}{||x||} - \frac{|f'(x)|}{||x||}\,\Big| \leq \min \{||f-f'||, ||f+f'||\}\leq \sqrt{2}\, \delta([f], [f']).\end{equation} We deduce from \eqref{hyp1} and \eqref{useful} that 
for every $k\in \N$, $\nu(\{x\in \PV; \delta(x,H)<a_{n_k}\})>\epsilon_0$. But since $H$ is closed in $\PV$ and $a_{n_k}\underset{k \rightarrow +\infty}{\longrightarrow} 0$, $\nu(\{x\in \PV; \delta(x,H)<a_{n_k}\})\underset{k\rightarrow +\infty}{\longrightarrow}  \nu(H)$. Thus $\nu(H)\geq \epsilon_0$ contradicting the non degeneracy of $\nu$. \\

\noindent Now we check that \eqref{es1} is equivalent to \eqref{es2}. Assume first that \eqref{es1} holds. We can then find a decreasing sequence $(a_n)_{n\in \N}$ in $(0,1)$  that converges to zero such that for every projective hyperplane $H$, 
 $$\forall k\in \N, \,\nu\{x\in \textrm{P}(V); \delta(x,H)\leq a_{k}\} < e^{-k}.$$
For every $k\in \N^*$, denote by  $U_k$   the   interval $[\frac{1}{a_k}, \frac{1}{a_{k+1}})$ with the convention $U_0=[1,\frac{1}{a_1})$.  
Let $\phi:[1,+\infty) \longrightarrow [1,+\infty)$ be any proper function such that   $\phi_{|_{U_k}}  \leq e^{k/2}$ for every $k\in \N$ (for instance affine on each $U_k$ with
 $\phi(\frac{1}{a_k}) = e^{(k-1)/2}$).  
  Let $H$ be a projective hyperplane and $A_k=  \{x\in \PV\setminus H; \frac{1}{\delta(x,H)}\in U_k\}$, $k\in \N$.  Since $(A_k)_{k\geq 0}$ covers 
$\PV\setminus  H$ and since $\nu$ is not degenerate on $\PV$, we deduce that 
 \begin{eqnarray}
 \int_{\PV}{\phi\left(\frac{1}{\delta(x,H)}\right)\,d\nu(x)} &=& \sum_{k=0}^{+\infty}{\int_{A_k}{\phi\left(\frac{1}{\delta(x,H)}\right)\,d\nu(x)}}\nonumber\\
 &\leq &  \sum_{k=0}^{\infty}{e^{-k}e^{k/2}} <+\infty.
 \nonumber\end{eqnarray}
 The finite sum above being independent of $H$,  the forward implication is proved. \\
Conversely, assume that \eqref{es2} holds and let  $C:=\sup_{H} \int_{\PV}{\phi\left(\delta^{-1}(x,H)\right)\,d\nu(x)} <+\infty$. \\
Let $\epsilon>0$. By properness of $\phi$ we can find $\eta>0$ such that $\phi(\frac{1}{t})>\frac{C}{\epsilon}$ for every $0<t<\eta$.  \\Hence, for every $t\in (0,\eta)$ and for every projective hyperplane $H$,  \\
$\nu \big\{x\in \PV; \delta(x,H)<t\big\}\leq \nu \big\{x\in \PV; \phi(\delta(x,H)^{-1})> \frac{C}{\epsilon}\big\}$. By Markov's inequality, we deduce that
for every $H$ one has that  $\nu \big\{x\in \PV; \delta(x,H)<t\} \leq \epsilon$ whenever $t\in (0,\eta)$. This proves the backward implication.  
\end{proof}

\begin{remark} Assume now that $\Gamma_{\mu}$ is strongly irreducible and proximal, so that the stationary measure $\nu$ is unique. \begin{enumerate}
\item When $\mu$ has an exponential moment,    Guivarc'h showed in \cite{Guivarch3} that  $\phi(x)=x^{\alpha}$ works for some $\alpha>0$ small enough. In other terms, $\nu$ has Holder regularity. In particular, $\nu$ has positive Hausdorff dimension. 
\item When  $\mu$ has a moment of order $p>1$,   Benoist and Quint proved in   \cite{benoist-quint-tcllineaire} that $\phi(x)=\left(\ln(x)\right)^{p-1}$ works. In particular, when $p\geq 2$, $\nu$ is log-regular (i.e.~ $\phi(x)=\ln x$ works). We note that proving the log-regularity of $\nu$ when $p=2$ was crucial for Benoist and Quint to prove the CLT for $\ln ||L_n||$. 
\item Lemma \ref{p0} shows that   such a function $\phi$ still exists when $\mu$ has a moment of order one. However, it  does not give an explicit   rate of growth for $\phi$. It would be interesting to determine such a rate. 
\item If $\Gamma_{\mu}$ is a non-elementary subgroup of  $\textrm{SL}_2(\R)$, more can be said about   the regularity of $\nu$ when $\mu$ has a moment of order one.   Indeed, using the   work of Benoist and Quint \cite[Section 5]{benoist-quint-tclhyperbolic} on central limit theorems on hyperbolic groups, one can deduce that the unique $\mu$-stationary probability measure on the projective line is $\log$-regular, even when $\mu$ has a moment of order one.  \end{enumerate}
\label{r11}\end{remark}

 We state now the estimates we will use. 
 
  \begin{prop} Assume that $\mu$ has a moment of order one and that $\Gamma_{\mu}$ is strongly irreducible and proximal. 
 Denote by $\mu^t$ the pushforward probability measure on $\textrm{GL}(V^*)$ of $\mu$ by the map $g\mapsto g^t$. Then 
 \begin{enumerate}

 \item For every $\epsilon>0$, $$\sup_{||v||=1} \p\left(\frac{||L_n v||}{||L_n||} \leq \exp(-\epsilon n)\right) \underset{n\rightarrow +\infty}{\longrightarrow} 0.$$
  \item There exists $C>0$ such that 
$$\sup_{x,y\in \textrm{P}(V)} \p\left( \delta(L_n \cdot x,L_n \cdot y) \geq \exp(-Cn)\right)  \underset{n\rightarrow +\infty}{\longrightarrow} 0.$$
\item There exists $C>0$ such that 
$$\sup_{ x\in \PV} \p\left( \delta(R_n \cdot x ,v_{R_n}^+) \geq \exp(-Cn)\right)  \underset{n\rightarrow +\infty}{\longrightarrow} 0.$$
 \item There exists $C>0$, a    random variable $Z$ with values in $\textrm{P}(V)$ of law the unique $\mu$-stationary probability measure on $\textrm{P}(V)$  such that $$ 
\sup_{x \in \textrm{P}(V)}{\p\left(\delta(R_n \cdot x,Z)\geq \exp(-Cn) \right)} \underset{n\rightarrow +\infty}{\longrightarrow} 0\,\,\,\,\textrm{and}\,\,\,\,\p\left(\delta(v_{R_n}^+,Z)\geq \exp(-Cn) \right)  \underset{n\rightarrow +\infty}{\longrightarrow} 0$$
 \item Similarly,     there exists  $C>0$, a  random variable $Z^*\in \textrm{P}(V^*)$ of law the unique $\mu^t$-stationary probability measure on $\textrm{P}(V^*)$ such that if $H_{L_n}^{-}:= [\ker(f_{L_n})]$, then 
 $$\p\left( \delta(f_{L_n}, Z^*)\geq \exp(-Cn)\right)  \underset{n\rightarrow +\infty}{\longrightarrow} 0,$$
 where $\delta$  denotes again, by abuse of notation,  the standard metric on $\textrm{P}(V^*)$. 
      \end{enumerate}
\label{prop-estimates} \end{prop}

  \begin{proof}
 We will use in all the proof that if $(A_n)_n$ and $(B_n)_n$ are two sequences of subsets of $\Omega$ such that $\p(A_n)=1-o(1)$ and $\p(B_n)=1-o(1)$, then $\p(A_n \cap B_n)=1-o(1)$. 
\begin{enumerate}
\item By \cite[Corollary 3.4  item (iii)]{bougerol-lacroix}, we know that for any sequence  $(v_n)_n$ in $V$ of norm one, 
$\frac{1}{n} \E \left(\ln ||L_n v_n||\right)\underset{n\rightarrow +\infty}{\longrightarrow} \lambda_1$. Hence   $\frac{1}{n} \E\left( \ln \frac{||L_n v_n||}{||L_n||}\right)  \underset{n\rightarrow +\infty}{\longrightarrow} 0$ for every such sequence $(v_n)_n$.   Thus, 
  $$\sup_{[v]\in \textrm{P}(V)}{\frac{1}{n}\E \left( \ln{\frac{||L_n||\,||v||}{||L_n v||}}\right) }\underset{n\rightarrow +\infty}{\longrightarrow}  0.$$
It is enough now to apply Markov's inequality in order to have the estimate of 1. 
\item Let $x=[v],y=[w]\in \textrm{P}(V)$. Without loss of generality $||v||=||w||=1$. We have by the definition of the metric $\delta$: 

\begin{equation}\forall g\in \textrm{GL}(V),\, \delta(g \cdot x ,g \cdot y)\leq \frac{||\bigwedge^2 g||}{ ||g||^2}\,\times \,\frac{||g||^2}{||g v||| \,||g w||}.\label{pol}\end{equation}
On the one hand, we know by the Guivarc'h-Raugi theorem \cite{guivarch-raugi} that  with our assumptions on the semi-group generated by the support of $\mu$, the first Lyapunov exponent is simple. Hence the following almost sure convergence holds:
 $$\frac{1}{n} \ln \frac{|| \bigwedge ^2 L_n||}{||L_n||^2} \underset{n\rightarrow +\infty}{\longrightarrow}  \lambda_2-\lambda_1<0.$$
We deduce that for $C:=(\lambda_1-\lambda_2)/2>0$, we have that 
\begin{equation}\p \left(  \frac{||\bigwedge ^2 L_n||}{||L_n||^2} \leq  \exp(-Cn) \right) = 1-o(1).\label{pol1}\end{equation}
On the other  hand, applying estimate 1.~for $\epsilon=C/4$ to  get that \begin{equation}\p \left( \frac{||L_n||^2}{||L_n v ||\,||L_n w ||} \leq \exp(Cn/2)\right) = 1-o(1).\label{pol2}\end{equation}
 Moreover the previous estimate is uniform in $v$ and $w$. 
Combining (\ref{pol}), (\ref{pol1}) and  (\ref{pol2}) and the  remark at the beginning of the proof, we get the desired estimate. 
\item Let  $x=[v]\in \textrm{P}(V)$ and $g\in \textrm{GL}(V)$. Observe  that $$\delta(v_{g}^+, g \cdot x)= \delta(e_1, a_g u_g \cdot x)= O\left(\frac{a_2(g)}{a_1(g)}\right) \times \frac{||g||\,||v||}{||g v||}.$$
 It is enough now to apply estimates 1.~and 2. 

\item   Let  $x\in \textrm{P}(V)$.  We know from \cite[Theorem 4.3]{bougerol-lacroix} that there exists a random variable $Z$ on $\textrm{P}(V)$  independent of $x$ of law the unique $\mu$-stationary probability measure on $\textrm{P}(V)$ such that the sequence of random variables 
$(R_n \cdot x)_{n\in \N}$ converges in probability to $Z$. Hence, there exists a non random subsequence $(n_k)_{k\in \N}$ such that   $(R_{n_k} \cdot x)_{k\in \N}$   converges almost surely to $Z$.  Fix now $n\in \N$ and denote by $C$ the positive constant given in estimate 2. On the one hand, we have by   Fatou's lemma that: 
\begin{equation}\p\left(\delta(R_n \cdot x, Z)> \exp(-Cn)\right)\leq \liminf_{k\longrightarrow + \infty}{\p(\delta(R_n \cdot x, R_{n_k} \cdot x)>  \exp(-Cn)) }.\label{ppl}\end{equation}
On the other hand, writing $R_{n_k} \cdot x=R_n \cdot (X_{n+1} \cdots X_{n_k}) \cdot x$ for all $n_k>n$ and using  the independence of the $X_i$'s, we get that for all $n_k>n$, 
 
\begin{equation}\p\left(\delta(R_n \cdot x , R_{n_k} \cdot x)>\exp(-Cn)\right)\leq \sup_{a,b \in \textrm{P}(V)}{\p\left(\delta(R_n \cdot a, R_n \cdot b)> \exp(-Cn)\right)}.\label{ppk}\end{equation}

Combining (\ref{ppl}) and (\ref{ppk}), we deduce that for every $n\in \N$, 
 $$ \p\left(\delta(R_n \cdot x, Z)> \exp(-Cn)\right) \leq \sup_{a,b\in \textrm{P}(V)}{\p\left(\delta(R_n \cdot a, R_n \cdot b)> \exp(-Cn)\right)}.$$
By estimate 2.~ and the fact that $R_n$ and $L_n$ have the same law for every $n\in \N$, we deduce that the  quantity above goes to zero as $n$ tends to infinity. This proves the first inequality. The second estimate then follows   item 3. 

\item Apply the previous estimate for the probability measure $\mu^t$ which satisfies the same  assumptions as $\mu$ (see for instance \cite[Chapter III, Lemma 3.3]{bougerol-lacroix}). 
 \end{enumerate}
\end{proof}

 \begin{remark} In parts 1,2,3,4  the speed of convergence is 
 \begin{itemize}
 \item exponential when $\mu$ has an exponential moment \cite{bougerol-lacroix}, \cite{Guivarch3}, \cite{aoun1}.
 \item of order $C_n$ for some sequence $(C_n)_{n\in \N}$ that satisfies   $\sum_n{n^{p-2} C_n}<+\infty$,   when $\mu$ has a moment of order $p>1$  \cite{benoist-quint-tcllineaire}. \end{itemize}
\end{remark}
 
 \begin{remark} The role of $R_n$ and $L_n$ is interchangeable in the statements of Proposition \ref{prop-estimates} except for estimates 4 and 5 where the result fails if we interchange $R_n$ and $L_n$. \label{r12}  \end{remark}

    \section{End of the proof of Theorem \ref{key}}
  The end of the proof is based on a usual trick about the asymptotic independence of the right and the left random walk. We refer for instance to  \cite[Section 2.3]{tutubalin},  \cite[Section 6]{vircer}, \cite[Theorem 1.5]{goldsheid-guivarch} and  \cite[Lemme 4.3]{aoun3} for a general statement. 
  \begin{proof}[Proof of Theorem \ref{key}:]  
Let $\mathcal{H}$ be the set of all projective hyperplanes of $\PV$. For every $H=\ker(f)$ and $H'=\ker(f')$ in $\mathcal{H}$, we denote for simplicity 
 $\delta(H,H'):=\delta([f],[f'])$. By estimates 3.~,  4.~ and 5.~ of Proposition \ref{prop-estimates}, there exist a random variable $Z\in \textrm{P}(V)$,   $C>0$, $n_0\in \N$
such that for every $n\geq n_0$: 
\begin{enumerate}
\item[i.] $$\p\left( \delta(v_{X_1 \cdots X_n}^+, Z) \geq  \exp(-Cn)\right)=o(1).$$
In particular, 
 $$\p\left( \delta(v_{X_1 \cdots X_n}^+, v_{X_1\cdots X_{\lfloor n/2 \rfloor}}^+) \geq  2\exp(-Cn)\right)=o(1).$$
\item[ii.] $$\p\left( \delta(H_{X_1\cdots X_n}^{-} ,  H_{X_{\lfloor n/2 \rfloor +1} \cdots X_n}^{-})\geq \exp(-Cn)\right) = \p\left( \delta( H_{X_n\cdots X_1}^{-} ,  H_{X_{n-\lfloor n/2 \rfloor } \cdots X_1}^{-}) \geq  \exp(-Cn) \right) = o(1).$$
\end{enumerate}
\noindent  The fact that the $X_i$'s are i.i.d~ is used in the left equality above, while the right one follows from Estimate 5. of Proposition \ref{prop-estimates}. 
 Fix now $\epsilon>0$.  We deduce that, for $n\geq n_0$, 
 \begin{eqnarray}
\p\left(\delta(v_{X_1\cdots X_n}^+, H_{X_1\cdots X_n}^{-}) \leq \epsilon \right) & \leq  & o(1)+ 
   \p\left( \delta\left(v_{X_1\cdots X_{\lfloor n/2 \rfloor}}^+, H_{X_1\cdots X_n}^{-} \right) \leq  \epsilon + 2\exp(-Cn)\right)\,\,\,\,\,\,\,\, \label{zz0}   \\
   & \leq &   o(1)+ \p\left( \delta\left(v_{X_1\cdots X_{\lfloor n/2 \rfloor}}^+, H_{X_{\lfloor n/2 \rfloor +1}\cdots X_n}^{-} \right) \leq \epsilon + 4\exp(-Cn)\right)\,\,\,\,\,\,\,\,\,\,\,\,\,\,\,\,\label{zz}\\
   & = & o(1) + \sup_{H \in \mathcal{H}}{\p\left(\delta(v_{X_1\cdots X_{\lfloor n/2 \rfloor}}^+, H)\leq \epsilon  + 4\exp(-Cn)\right)}\,\,\,\,\,\,\,\,\label{zz1}\\
  & \leq & o(1)+ \sup_{H \in \mathcal{H}}{\p\left(\delta(Z, H)\leq \epsilon + 5\exp(-Cn)\right)}. \,\,\,\,\,\,\,\,  \label{zz2}\end{eqnarray}
Estimates  \eqref{zz0} and \eqref{zz2} follow  immediately from estimate i.~ at the beginning of the proof.   In line \eqref{zz}, we used estimate ii.~ above and identity \eqref{useful}. 
  Identity  \eqref{zz1} is due to the independence of the $(X_i)_i$'s.\\\
  Observe now that, by compactness of $\mathcal{H}$,   the following convergence holds for   $\epsilon>0$ fixed:
  $$\sup_{H \in \mathcal{H}}{\p\left(\delta(Z, H)\leq \epsilon + 5\exp(-Cn)\right)} \underset{n\rightarrow +
  \infty}{\longrightarrow} \sup_{H \in \mathcal{H}}{\p\left(\delta(Z, H)\leq \epsilon\right)}.$$
  Since the law of the random variable $Z$ is the unique   $\mu$-stationary probability measure $\nu$,  we deduce from   \eqref{zz2} that for every $\epsilon>0$,  
\begin{equation}\limsup_{n\rightarrow +\infty}{\p\left(\delta(v_{X_1\cdots X_n}^+, H_{X_1\cdots X_n}^{-}) \leq \epsilon \right) }\leq \sup_{H \in \mathcal{H}}\nu\left\{x\in \PV; \delta(x,H)\leq \epsilon\right\}.\label{proofo}\end{equation}
  Applying Lemma \ref{p0} ends the proof. 
  \end{proof}

    The proof of Theorem \ref{main} (combine \eqref{toprox}, \eqref{proofo0} and \eqref{proofo})  shows the following result which   makes clear the link between the regularity of stationary measures on projective space (see Remark \ref{r11}) and the speed of convergence   as $\epsilon \rightarrow 0$ of the function $\epsilon \mapsto \limsup_{n \rightarrow +\infty}{\p
 \left( \frac{\rho(L_n)}{||L_n||} \leq \epsilon \right)}$. 

  \begin{theo}
  Let $\mu$ be a probability measure on $\textrm{GL}(V)$ such that $\mu$ has a moment of order one and such that 
  $\Gamma_{\mu}$ is strongly irreducible. Let $p$ be the proximality index
   of $\Gamma_{\mu}$. Then   there exists $C=C(\Gamma_{\mu})\geq 1$ such that 
   for every $\epsilon>0$, 
 $$\limsup_{n\rightarrow +\infty}{\p\left( \frac{\rho(L_n)}{||L_n||} \leq  \epsilon  \right)}\leq   \sup_{H \,\textrm{projective hyperplane}}
 \nu\left\{x\in \textrm{P}(V); \delta(x,H)\leq 2 C \epsilon^p\right\}  \underset{\epsilon \to 0}{\big\downarrow} 0,$$
 where  $\nu$ is the unique $\mu$-stationary measure on the projective space of some irreducible and proximal representation\footnote{Assume we are working in   characteristic zero. 
 Then $\nu$ is the unique $\mu$-stationary probability measure on   the projective space  $\textrm{P}(\bigwedge^ p V)$
 whose cocycle average is maximal (i.e. 
 $\iint {\ln \frac{||\bigwedge^p g\, v||}{||v||}\, d\mu(g)d\nu([v])}=p\lambda_1(\mu)$).} of $\Gamma_{\mu}$.  $\qed$
\label{mainprecise}\end{theo}

 We end by justifying Remark \ref{rem-strongly}.  The setting is borrowed from \cite[Example 3.5]{breuillard-cagri}.

  \begin{example}\label{notconv}
Let $\lambda>1$, $\sigma=\begin{pmatrix}0& -1\\
1 & 0\end{pmatrix}$ and $a=\begin{pmatrix}\lambda & 0\\
0& \lambda^{-1}\end{pmatrix}$. Let   $\mu$ be a probability measure on   $S:=\{\sigma, \sigma a\}$ with full support, $\theta:=\mu(\sigma a)\in (0,1)$, $(X_i)_{i\in \N}$ a sequence of i.i.d~random variables with law $\mu$ and consider the random walk $L_n=X_n\cdots X_1$ on $\textrm{SL}_2(\R)$ with respect to $\mu$. We claim that $\lambda_1(\mu)=0$, $\frac{\ln \rho(L_n)}{n}\to 0$ almost surely  but that  $\frac{\ln \rho(L_{n})}{\sqrt{n}}$ does not converge in distribution. 
On the contrary,     $\frac{\ln ||L_{n}||}{\sqrt{n}}$ does converge in distribution but to the maximum of two (dependent) Gaussian distributions. The semi-group $\Gamma_{\mu}$ is actually a subgroup of $\textrm{SL}_2(\R)$ and it is irreducible but not strongly irreducible 
(as the union of the coordinate axis of $\R^2$ is $\Gamma_{\mu}$-invariant). Let us check our claims. Denote by 
 $H\subset \Gamma_{\mu}$   the subgroup of diagonal matrices.   Note that   $H$ is a  subgroup of $\Gamma_{\mu}$ of index two and that the matrices in the coset $\sigma H$ have spectral radius equal to one. On the one hand, since $\mu$ is supported in the coset $\sigma H$, we deduce that $L_{2n+1}\in \sigma H$ for every $n$ and therefore that $\rho(L_{2n+1})=1$ for every $n$. In particular, the only possible limiting distribution of  $\frac{\ln \rho(L_{n})}{\sqrt{n}}$ is the Dirac mass at $0$.  On the other hand,    writing $X_i=\sigma a^{\epsilon_i}$ with   $(\epsilon_i)_{i\in \N}$ a sequence of independent
    Bernoulli random variables $\mathcal{B}(\theta)$, and using the relation $a^{k}\sigma=\sigma a^{-k}$ for $k\in \Z$, we get for $Y_i:=\epsilon_{2i-1}-\epsilon_{2i}$, 
 $$L_{2n}=(-1)^n a^{S_n}\,\,\,\,\textrm{with}\,\,\,\,S_n=\sum_{i=1}^n{Y_i}.$$
 The sequence $(Y_i)_{i\in \N}$ is a sequence of i.i.d.~random variables  on $\{-1,0,1\}$   with $\p(Y_i=-1)=\p(Y_i=1)=\theta (1-\theta)$, so that $S_n$ is a centered random walk on $\Z$. Thus, by the classical central limit theorem,  $\frac{\ln \rho(L_{2n})}{\sqrt{n}}=\ln \lambda\, \frac{|S_n|}{\sqrt{n}}$ converges in distribution to $|X|$ with    $X$ being a non degenerate Gaussian distribution. Thus $\frac{\ln \rho(L_n)}{\sqrt{n}}$ does not converge in distribution. The other claims follow readily from the discussions above. 
 \end{example}

\end{document}